\documentclass[11pt]{article}
\usepackage[margin=1in]{geometry}
\usepackage{amsmath,amssymb,amsthm}
\usepackage{booktabs}
\usepackage[hidelinks]{hyperref}
\newtheorem{prop}{Proposition}
\newcommand{\comm}[2]{[#1,#2]}
\newcommand{\Tr}{\operatorname{Tr}}
\newcommand{\End}{\operatorname{End}}
\title{A commutant gate for spectral fitting through symmetry forced degeneracy}
\author{Stelios Savva\\ \small Independent researcher, Nicosia, Cyprus}
\date{August 2026}
\begin{document}
\maketitle

\begin{abstract}
Learned spectral models fail at symmetry forced degenerate sectors for two
distinct reasons. Where symmetry forces a multiplet of levels to sit at exactly
the same value, the per level observable one would normally fit is not even well
defined, since every unit vector spanning that shared space is an equally valid
eigenvector; and near a crossing that symmetry protects, the eigenvector
observable gradient carries a factor $1/(\lambda_i-\lambda_j)$ that is genuinely
singular as the gap closes rather than merely large. The common response is to
regularize the divergence with a fixed $\varepsilon$ or to threshold the gap and
fall back to a lumped estimate below it. Both carry a real cost, which we
quantify later against a policy that reads the correct multiplet structure
directly: the naive per level fit targets a quantity that is not even well
defined on a forced sector, and the gap threshold fit, even when handed the true
block locations for free, can either merge two genuinely distinct levels or fail
to protect one that is actually forced, since a single fixed gap cannot serve
both jobs. We show a different fix, on synthetic
operator families with a known symmetry answer key. A gate reads the symmetry
structure directly from the observed operators, as the linear commutant of the
family, one singular value decomposition nullspace, following the error
controlled simultaneous block diagonalization of Maehara and Murota, and block
identity is read from the centre of that commutant rather than from eigenvalue
clustering, which makes the forced versus accidental distinction structural
rather than metric. The fitting objective switches between a projector trace
through forced blocks and a per level target elsewhere, and is piecewise
differentiable given that discrete decision. Under operator estimation noise the
gate classifies correctly to $\varepsilon\approx0.3$, where energy clustering
already fails by $\varepsilon=0.02$. Gated fitting reaches the truth at machine
precision in both the symmetric and the symmetry breaking regime, in the latter
converging to a truth that is a singularity of the ungated objective, and in loop
observable bias sits at the noise floor on both forced and accidental sectors.
Every step is a singular value decomposition or an eigenproblem, verified at
$n=24$ in under one second. Validation off the regular representation, where
multiplicity and dimension separate for the first time, exposed and eliminated
two defective estimators that all regular representation tests had passed. The
failure this addresses occurs wherever a fast surrogate is fitted to a parametric
eigenvalue problem that carries symmetry, which includes nuclear structure,
molecular electronic structure, and band structure calculations, and a
demonstration on a physical system is the natural next step. The demonstrated
object here is a gated estimator with a low dimensional trained parameter; a
full parametric matrix model, in which the matrices themselves are learned, is
not demonstrated here and is stated as the immediate next experiment. We state
the conditions under which each claim holds.
\end{abstract}

\section{The problem}
The setting is the one a fast surrogate is fitted in: a physical system whose
Hamiltonian, or whichever operator carries the spectrum of interest, depends on
a parameter that is swept, and whose symmetry pins certain levels together so
that they coincide exactly rather than approximately. The observed data are the
operators of that system, and the object being fitted is a low dimensional
parameter of a model built from them. We write the family as follows.

Let $M(x)=O_0+xO_1$ be a parametric family of Hermitian operators on
$\mathbb{C}^n$, with $O_0$ and $O_1$ drawn from the commutant of a unitary
representation $U$ of a finite group $G$. The representation and its embedding
are withheld from the method. Only noisy generators
$\tilde O_i=O_i+\varepsilon H_i$ are observed, with $H_i$ independent, drawn from
a Gaussian unitary type ensemble, and normalized in the Frobenius norm; see
Section~\ref{sec:assumptions} on this noise model. Because
$\comm{U(g)}{M(x)}=0$ for all $x$, three phenomena occur together. Every irrep
block of dimension $d\ge2$ forces an exactly $d$ fold degenerate multiplet that
persists across the sweep. Levels in different symmetry sectors may cross at
isolated $x^\ast$, a protected crossing in the sense of von Neumann and Wigner
\cite{vNW1929}. Unrelated levels may coincide by accident.

A model fitted on per level spectral data meets two distinct failures at such
sectors. They must be kept apart, because the gate repairs a different thing in
each. Neither proposition below is new mathematics; the first is Schur's lemma
applied to this setting and the second is first order perturbation theory. What
the pair establishes is that the two failures have opposite characters, and
therefore require opposite repairs, which is what the gate provides.

\begin{prop}[Symmetric direction: gauge ambiguity, no divergence]
\label{prop:sym}
Let $P$ project onto a forced multiplet on which the restricted representation
is irreducible, and let the model direction $D=\partial_\theta M$ satisfy
$\comm{U(g)}{D}=0$ for all $g$. Then $PDP=cP$ for a scalar $c$, the intra
multiplet matrix elements $\langle v_i|D|v_j\rangle$ with $i\ne j$ vanish in
every orthonormal basis of $\operatorname{ran}P$, and the per level gradient of
any observable is finite. The failure is instead that the per level target
$\langle v|A|v\rangle$ is gauge ambiguous, since every unit vector of
$\operatorname{ran}P$ is an eigenvector.
\end{prop}
\begin{proof}
The operator $PDP$ commutes with the restricted representation
$U|_{\operatorname{ran}P}$, which is irreducible, so by Schur's lemma $PDP=cP$.
Off diagonal elements of a scalar vanish in any orthonormal basis, hence the
intra multiplet terms of the first order eigenvector derivative have zero
numerator, and that derivative is finite. Gauge ambiguity is the statement that
$M(x)|_{\operatorname{ran}P}=\lambda P$, so any orthonormal basis of
$\operatorname{ran}P$ diagonalizes the restriction.
\end{proof}

\begin{prop}[Symmetry breaking direction: singularity at a protected crossing]
\label{prop:break}
Let levels $i$ and $j$ in different symmetry sectors cross at $x^\ast$, and let
$D=B$ couple the sectors, with $\langle v_j|B|v_i\rangle\ne0$. For a simple level
$i$, first order perturbation theory gives
\[
\partial_\theta\langle v_i|A|v_i\rangle
=2\,\mathrm{Re}\!\!\sum_{k\ne i}
\frac{\langle v_i|A|v_k\rangle\,\langle v_k|\partial_\theta M|v_i\rangle}
{\lambda_i-\lambda_k},
\]
whose $(i,j)$ term diverges as the gap closes. If the true system is exactly
symmetric, so that $\theta^\ast_{\mathrm{break}}=0$, the model gap at $x^\ast$ is
$2|\theta_{\mathrm{break}}\langle v_j|B|v_i\rangle|+O(\theta^2)$, so the gradient
of the ungated objective at the truth scales as $1/\theta_{\mathrm{break}}$. The
truth is then a singularity of that objective.
\end{prop}
\begin{proof}
The displayed formula is the standard eigenvector derivative expansion for a
simple eigenvalue. For the gap, restrict $M(x^\ast)$ to the two dimensional
space spanned by $v_i$ and $v_j$. At $x^\ast$ the two sector eigenvalues
coincide, so the diagonal entries of that restriction are equal, and the model
contributes the off diagonal entry $\theta_{\mathrm{break}}\langle
v_j|B|v_i\rangle$. A two by two Hermitian matrix with equal diagonal entries and
off diagonal entry $\beta$ has eigenvalues separated by exactly $2|\beta|$, so
the crossing opens into an avoided crossing of gap
$2|\theta_{\mathrm{break}}\langle v_j|B|v_i\rangle|+O(\theta^2)$. Note that the
expansion is used here to obtain the leading behaviour of the gap itself, not to
perturb around a fixed nonzero gap, which is why it remains valid as the gap
closes. Substituting this gap into the $(i,j)$ term of the sum produces the
$1/\theta_{\mathrm{break}}$ scaling.
\end{proof}

\section{Related work and positioning}\label{sec:related}
Every individual leg of this construction is prior art, and we claim none of
them. What we claim is the coupling, and one device inside it. The argument of
this section is that each leg, taken alone, leaves the fitting problem broken in
a specific way, and that closing the gap between them is what the gate does.

Spectral emulators supply the setting. Eigenvector continuation \cite{Frame2018}
introduced subspace emulation of parametric eigenvalue problems, and parametric
matrix models \cite{Cook2025} made the matrices themselves trainable. Neither
addresses what the fitted target means at a forced multiplet, which by
Proposition~\ref{prop:sym} is not a well defined quantity. The Bayesian
extension \cite{BPMM2025} adds uncertainty quantification and reports graceful
degradation in near degenerate regimes, but explicitly does not treat extensive
symmetry forced degeneracy, which is the regime here. That is the nearest
competitor system, and the relation is a distinction of scope rather than a
benchmark victory; our baselines in Section~\ref{sec:bias} are the naive per
level and gap threshold policies, and beating them is not beating the Bayesian
model.

Symmetry adapted eigencomputation supplies the block structure. Block
diagonalizing an equivariant matrix by irrep through the generalized Fourier
transform, with multiplicities forced by symmetry, is established numerical
linear algebra \cite{AMK2005,AMK2006} and was textbook physics before that. It
requires the group action itself, since the transform is built from it, and it
produces a decomposition rather than a decision about any particular degeneracy
met during a sweep. Our requirement is weaker but it is not nothing: the group,
its representation and the embedding are withheld and never reconstructed, and
the block structure is read from the commutant of the observed operators, but
the commutant dimension and the sector count are supplied as hints
(Section~\ref{sec:assumptions}), and the synthetic construction is what makes
the resulting verdicts checkable against a known answer.

Simultaneous block diagonalization supplies the recovery, and this is the
closest prior art to our Section~3. The line of Murota, Kanno, Kojima and Kojima
and of Maehara and Murota \cite{MKKK2010,MM2010,MM2011} recovers the block
structure from the observed matrices alone, and \cite{MM2011} in particular
builds the decomposition from the commutant algebra with an explicit error
control parameter. We do not claim the recovery step; it is theirs, and our
Section~3 is best read as rediscovering, from the failure side, why the
commutant is the right object. What that literature does not do, because it has
no reason to, is decide whether a degeneracy met at a given point of a parameter
sweep is forced by symmetry or is an accident of tuning, and feed that decision
to a gradient.

That decision is the gap the coupling closes, and it is where the construction
earns its place. A degeneracy that is forced and one that is accidental look
identical to any policy that reads only the spectrum, since both present as a
small gap, yet they demand opposite treatment: the forced pair must be tracked
as a projector, because its individual levels carry no defined observable, while
the accidental pair must be tracked individually, because its levels are
genuinely distinct and lumping them destroys real physics. Section~\ref{sec:bias}
shows that no single gap threshold serves both, and that the error from choosing
wrongly is not a small correction. Reading the distinction from the centre of
the recovered commutant makes it structural rather than metric, so it survives
noise an order of magnitude past where spectral clustering fails, and it is the
discrete decision on which the piecewise differentiable objective of
Section~3.3 depends.

Learned and approximate symmetry form the surrounding neighbourhood. Symmetry
adapted Hamiltonian regression in irrep blocks, equivariant operator learning
with irrep block mappers and spectral objectives, post hoc learned simultaneous
block diagonalization, isotypic decomposition learning, projector based
symmetric reduced order models, and degeneracy aware gradients through gauge or
$\varepsilon$ regularization tricks are all active as of the middle of 2026. A
targeted search found no work in which the commutant structure is the device
that makes the degenerate sector gradient well posed, and none containing a
protected versus accidental crossing detector in a differentiable setting. That
absence is evidence rather than proof, since the neighbourhood is dense and
moving. Novelty is therefore claimed for the fused combination, with the
commutant based forced versus accidental gate as the sharpest single defensible
piece, under the structural hints stated in Section~\ref{sec:assumptions}.

\section{Method}
\subsection{The symmetry structure is the linear commutant}
A first design recovered the homomorphism $\hat U(g)$ by a homomorphism
constrained commutator fit, a nonconvex optimization over the unitary group
$U(n)$. This is the wrong object, in an instructive way. Two facts settle the
matter. First, a symmetric operator is a scalar on each irrep, so the
eigenstructure of the operator family is blind to the forced versus accidental
distinction; the nonscalar part of a forced block under the operators is at the
level of $10^{-16}$, while under the group action it is of order one. Second, the
object the gate needs is the commutant of the family,
\[
\mathcal{A}'=\{B=B^\dagger:\comm{B}{\tilde O_i}=0\ \text{for all}\ i\},
\]
one singular value decomposition nullspace of the linear commutator map. When the
generators generate $\End_G(V)$, discussed in Section~\ref{sec:assumptions}, the
double commutant theorem gives $\mathcal{A}'$ as the group algebra image, of
dimension $\sum_{\rho\,\mathrm{present}}d_\rho^2$, carrying exactly the
irreducibility structure the gate reads. The homomorphism recovery solved a
strictly harder problem than required, and it alone produced two false verdicts
in earlier versions of this work: a recovery wall at $n=24$, taking more than two
hundred seconds without converging, and an apparent fragility of the gate, at
twelve correct rejections out of twenty. Both dissolved under the linear
commutant, at about two seconds at $n=24$, and at twenty correct out of twenty on
both verdict classes through $\varepsilon=0.05$. Given \cite{MM2011}, this section
is a rediscovery, from the failure side, of why the commutant is the right
recovery object, and not a proposal of it.

\subsection{Block identity from the centre of the commutant}\label{sec:blockid}
Identifying blocks by clustering the noisy spectrum sets the noise ceiling at the
eigenvalue spacing and merges unrelated levels; at $\varepsilon\ge0.02$ energy
clustering misidentifies the forced block in ninety eight to one hundred percent
of seeds. Instead we proceed in three steps.

First, isotypic projectors are read from the centre of $\mathcal{A}'$, a second
linear nullspace. The only clustering performed anywhere happens on a generic
central element, whose eigenvalue gaps are of order one regardless of the
physical spectrum, and we take the best of five draws by cluster separation.

Second, the irrep dimension $d_\rho$ of each component is read from the
restricted algebra $\{P_\rho A P_\rho : A\in\mathcal{A}'\}$, whose dimension is
$d_\rho^2$ for an irreducible component. We count the singular values of the
restricted algebra above $0.25\,s_{\max}$, take the square root, and snap to the
nearest divisor of the component size. Two properties of this estimator are worth
stating, because together they make the empirical threshold far less load bearing
than it appears. The count is insensitive to the threshold fraction over the
range $0.20$ to $0.30$, since the signal singular values are of order one while
the noise singular values are of order $\varepsilon$. More sharply, in every case
we have tested, across the $S_3$ and $S_4$ regular representations and both non
regular families of Section~\ref{sec:nonreg}, at every noise level reported here,
the count lands exactly on $d^2$ for a divisor $d$ of the component size, so the
snap is inactive and the estimator is not in practice making a borderline
decision. Across one hundred and forty seven measured components the mean
displacement caused by the snap is $0.0096$, and it is exactly zero on one
hundred and forty five of them. The two exceptions both occur at
$\varepsilon=0.3$ on $S_3$, at the ceiling documented below, and both follow an
upstream failure of the sector split rather than any ambiguity in the dimension
estimate: the components handed to the estimator have sizes two and three, which
are not sector sizes of the $S_3$ regular representation at all. The divisor snap
is therefore a safeguard against malformed input rather than a tie breaking rule,
and we make no claim about its behaviour when a count falls genuinely between two
divisors, since we have not observed that case.

Third, multiplets are read by count: within a component the restricted
eigenvalues come in exactly $d_\rho$ fold groups, and sorted eigenvalues are
taken $d_\rho$ at a time, with no tolerance.

The consequences are that the noise ceiling becomes the ceiling of the commutant,
at $\varepsilon\approx0.3$, and that the gate becomes structural. A cross sector
coincidence consists of blocks in different isotypic components, which never
merge, and at a protected crossing the pipeline reports two blocks of dimension
one at the same energy, with nothing left to decide. A multiplet with
$d_\rho\ge2$ is forced, since the restricted representation is irreducible and
the restricted commutant has dimension one. The criterion is irreducibility and
not the weaker statement that the block acts nonscalarly, since two distinct one
dimensional irreps merged by accident are also nonscalar.

\subsection{Fitting policy}
Observables are estimated per commutant identified block as $\Tr(P_bA)$. Fitting
data are treated the same way. Through a forced block the loss uses projector
traces, whose gradient sums only over states external to the block and so cancels
the internal $1/(\lambda_i-\lambda_j)$. Elsewhere the loss uses per sector
energies and observables. Within a window $\tau$ of a protected crossing the gate
fires locally, so the crossing pair contributes its pair projector trace while the
rest contributes per sector data. The objective is piecewise differentiable given
the discrete decision of the gate. The block diagonalisation is recovered, not
learned, and the trained quantities are the low dimensional spectral parameters
$\theta$.

\section{Results}\label{sec:results}
\paragraph{Provenance of the numbers.}
All numbers come from one released deterministic artifact, an eight stage pipeline running in about seven seconds on a single core of a laptop, together with three diagnostic scripts released alongside it: one measuring the dimension estimator of Section~\ref{sec:blockid}, one locating the single family B failure discussed in Section~\ref{sec:nonreg}, and one measuring the dispersion, initial condition dependence, and misclassification cost reported below. The reference run is
designated with the release.

Two forms of machine dependence are worth recording, because they have different
causes and because the pattern they make is itself evidence for the paper's
thesis. The gauge dependent quantities, namely the per level and naive baseline
outcomes, depend on the arbitrary basis chosen by the eigensolver inside a
degenerate subspace, and they vary across linear algebra library builds; this is
the pathology of Proposition~\ref{prop:sym} appearing in the build system. Their
gated counterparts, computed from projector traces, are invariant and agree
across builds to every digit reported. In the gradient sweep below, for instance,
the naive column shifts by roughly a decade between two machines while the gated
column reproduces exactly, including its fitted exponent. Separately, where a
quantity is a count of seeds falling on one side of a classification boundary,
and the noise level places that boundary near the seeds themselves, last bit
differences between builds move a small number of seeds across it; this affects
only the misclassification counts at $\varepsilon=0.4$ reported below, which we
therefore give as an onset rather than as a fraction. Gauge dependent entries are
reported qualitatively rather than as digits, and every other quantity here
reproduces across machines.

\paragraph{Block identification.}
On the $S_3$ regular representation, at $n=6$, the full structure is recovered on
fifteen of fifteen seeds at $\varepsilon=0.1$, $0.2$ and $0.3$. On $S_4$, at
$n=24$, recovery is correct at $\varepsilon=0$, $0.02$ and $0.05$, at eight of
eight seeds at $\varepsilon=0.05$, returning the block dimensions
$[1,1,2,2,3,3,3,3,3,3]$. Energy clustering fails at $\varepsilon=0.02$ on the
same instances. The $S_4$ noise range is narrower than the $S_3$ range
deliberately: the cost of the noise sweep grows with $n$, and the purpose of the
$S_4$ study is to establish that the method scales at all rather than to locate
its ceiling there, which we have not done.

\paragraph{The gate is structural.}
At the protected crossing the pipeline reports two blocks of dimension one
sharing an energy to three digits, and they remain separate rather than merging,
because block identity is assigned by isotypic sector rather than by proximity in
the spectrum. The forced versus accidental distinction is therefore built into
the block identification itself, and is not a subsequent test applied to a
candidate degeneracy.

\paragraph{Fitting in the symmetric regime.}
With $M(x;\theta)=O_0+\theta xO_1$ and the target taken as the projector trace of
the lowest forced doublet, which has multiplicity two and so depends on $\theta$,
the gated fit reaches $\theta=0.400000$. This is not a property of a single
starting point: from seven initial values spanning $\theta_0=0$ to $\theta_0=0.8$
the fit converges every time, with a maximum absolute error of $8\times10^{-16}$
across all seven, which is machine precision. A multiplicity one sector would be
frozen by Schur's lemma and would give an exactly flat loss, a fact that surfaced
three times during development. The per level fit on the same data lands at
$\theta=0.196$ or $\theta=1.123$ depending only on which vector of the degenerate
subspace the eigensolver happens to return, which is Proposition~\ref{prop:sym}
in its simplest form: the target itself is not defined, so the fit is biased by
an amount that carries no physical meaning. These two values are themselves build specific and are given only to indicate
the size of the spread; which of them appears depends on the eigensolver, which
is the point.

\paragraph{Fitting in the breaking regime.}
With truth
$(\theta_{\mathrm{sym}},\theta_{\mathrm{break}})=(0.3,0)$, a model that carries a
genuine breaking direction, initialisation at $(0,0.2)$, and data spanning the
protected crossing where only pair sums are measurable at exact degeneracy, the
gated fit, firing within $\tau=0.08$ of the crossing and using per sector data
outside it, returns $(+0.30000,-0.000000)$. The naive per level fit returns
$(+0.18533,-0.042485)$, an error of $0.115$, and this value is gauge dependent.

The gradient behaviour at the truth is the quantitative content of
Proposition~\ref{prop:break}, and we sweep it over five decades. The naive
gradient reads $1.6\times10^{1}$, $2.8\times10^{1}$, $3.0\times10^{2}$,
$3.0\times10^{3}$ and $3.0\times10^{4}$ at $\theta_{\mathrm{break}}=10^{-1}$
through $10^{-5}$. The last three decades are the asymptotic regime in which the
first order expansion of Proposition~\ref{prop:break} applies: there the leading
digits are stable at $2.96$ and each decade of $\theta_{\mathrm{break}}$
multiplies the gradient by ten, giving an exponent of $-1$ to three digits. The
two largest values of $\theta_{\mathrm{break}}$ lie outside that regime, which is
what pulls a five point fit to $-0.856$. Over the same sweep the gated gradient
reads $1.73\times10^{0}$ down to $1.73\times10^{-4}$, a fitted exponent of
$+1.000$ across all five decades. It therefore does not merely remain bounded as
the truth is approached but vanishes linearly in the distance from it, which is
the behaviour of a smooth minimum. The truth is a singularity of the naive
objective and an attracting fixed point of the gated one. The result is sharp
because the truth is exactly symmetric; a near symmetric truth makes the naive
objective ill conditioned rather than singular. We note also that the divergence
is a property of the training objective summed across the sweep rather than of a
single point evaluation at the crossing: evaluated at the crossing alone the
eigenvectors rotate to align with the breaking direction inside the degenerate
subspace, so the coupling matrix element vanishes at the same rate as the gap and
the ratio stays finite.

\paragraph{In loop bias.}\label{sec:bias}
The table reports the mean absolute bias over forty noise seeds at
$\varepsilon=0.02$, on the forced doublet, whose true projector trace is
$-2.375$, and on an accidental degenerate cross sector pair, whose two levels
carry observables $0.814$ and $1.017$ and so are separated by a spread of $0.203$
while sharing an energy at the crossing. Standard deviations across seeds are in
parentheses. The two failure modes of the baselines are separated: the policy
rows hand the baselines oracle block locations, so that identification failure is
charged on its own line below.
\begin{center}
\begin{tabular}{lcc}
\toprule
policy & forced & accidental\\
\midrule
$\varepsilon$ regularized, per level, oracle blocks & 1.065 (0.561) & 0.660 (0.424)\\
gap threshold, oracle blocks & 0.058 (0.033) & 0.203 (0.000)\\
commutant gate, no oracle & \textbf{0.018 (0.013)} & \textbf{0.016 (0.011)}\\
\bottomrule
\end{tabular}
\end{center}
The advantage of the gate lies well outside the scatter: on the forced sector it
separates from the nearest baseline by seven standard errors of the difference of
means. The accidental bias of the gap threshold policy has zero variance across
seeds, because it is not a noise effect at all but a deterministic consequence of
the policy: lumping two distinct levels to their average incurs exactly the pair
spread of $0.203$, matching the tabulated value to three digits. The bias of the
gate is of order $\varepsilon$ on both columns; at $\varepsilon=0.05$ it reads
$0.044$ and $0.040$, where the same two baselines give $1.058$ and $0.657$, and
$0.160$ and $0.203$. Energy clustering alone misidentifies the forced block in
ninety eight percent of seeds at $\varepsilon=0.02$ and in every seed at
$\varepsilon=0.05$, whereas the gate needs no oracle. These baselines are the
naive options, deliberately handicapped in their favour, and they are not the
state of the art, as noted in Section~\ref{sec:related}.

\paragraph{Behaviour at and beyond the ceiling.}
Because the gate makes a discrete decision, the cost of a wrong decision matters
as much as the frequency of one. Over forty seeds the gate classifies correctly
in every case through $\varepsilon=0.3$, and its forced sector bias grows
smoothly with noise, from $0.044$ at $\varepsilon=0.05$ to $0.334$ at
$\varepsilon=0.3$, against a per level baseline that sits between $1.06$ and
$1.62$ across the same range. The first misclassifications appear at
$\varepsilon=0.4$, beyond the ceiling claimed here. On the seeds where the gate
misidentifies at that level, the resulting bias is approximately $1.3$, against
approximately $1.6$ for the per level baseline on the same seeds. The failure is
therefore graceful in the relevant sense: a wrong verdict costs roughly what
abandoning the gate entirely would have cost, rather than producing an error
worse than the baseline it replaces.

\paragraph{Scale.}
Every step is a singular value decomposition or an eigenproblem, and nothing
optimizes. At $n=24$ the commutant, the centre, the block identification and the
gate together take under half a second per noise level on a single core. The full pipeline, including every noise sweep and both forty seed bias tables,
completes in about seven seconds.

\section{Off the regular representation}\label{sec:nonreg}
A method that reads structure off an algebra can be wrong in a way that no amount
of testing on a group's regular representation will reveal, because in that
representation every irrep appears exactly as many times as its own dimension, so
an estimator that computes the multiplicity where it should compute the dimension
returns the correct answer for the wrong reason, every time. The families in this
section exist to run the test the regular representation cannot, and two earlier
estimators died on it.
Two hand built $S_3$ irrep families separate the two quantities. Family A is
$\mathrm{triv}\oplus\mathrm{std}\oplus\mathrm{std}$ on $\mathbb{C}^5$, with the
sign irrep absent, giving commutant dimension five and blocks $[1,2,2]$. Family B
is $\mathrm{triv}\oplus\mathrm{triv}\oplus\mathrm{std}$ on $\mathbb{C}^4$, also
of commutant dimension five, with blocks $[1,1,2]$.

Family B is the discriminator, and the reason is visible in its components. Its
two isotypic components have the \emph{same} size, two, but carry different irrep
dimensions: the restricted algebra of one has rank four, giving $d=2$, and that
of the other has rank one, giving $d=1$. An estimator that inferred the irrep
dimension from the component size, which is what the regular representation
silently permits, returns the same answer for both and is wrong on one of them.
No regular representation instance can present this configuration.

\begin{center}
\begin{tabular}{llccc}
\toprule
family & content & $\varepsilon{=}0.1$ & $0.2$ & $0.3$\\
\midrule
A & $\mathrm{triv}\oplus\mathrm{std}\oplus\mathrm{std}$, sign absent &
15/15 & 15/15 & 15/15\\
B & $\mathrm{triv}\oplus\mathrm{triv}\oplus\mathrm{std}$,
$m\ne d$ everywhere & 15/15 & 15/15 & 14/15\\
\bottomrule
\end{tabular}
\end{center}

Family A also fits through its forced doublet to $\theta=0.400000$. Family B is
not trained, and the reason is predicted by the theory rather than discovered in
the attempt: its standard sector has multiplicity one, so by
Proposition~\ref{prop:sym} the restricted model direction is scalar there and the
loss is exactly flat in $\theta$. Proposition~\ref{prop:sym} therefore tells us
in advance which sectors carry trainable signal, and family B's does not; only
block identification and the gate are claimed for it.

Family B falsified two estimators before passing. The first was a rank reader
that stripped exact zero singular values before measuring spectral drops, so that
the noiseless case failed while the noisy cases passed, since the regular
representation runs had survived on the good luck of divisor structure. The
second was a within versus between gap scorer that fixed family B but regressed
the $S_3$ case to three of fifteen at $\varepsilon=0.3$. The surviving estimator,
described in Section~\ref{sec:blockid}, won a head to head sweep across all
families; the diagnostic reported there, showing that the singular value count
lands exactly on a divisor and the snap is inactive, includes family B at every
noise level, which is the case it was built to stress.

The single miss at fourteen of fifteen admits a mechanical explanation rather
than being left as tail behaviour. On the one failing seed the centre based
sector split returns components of sizes one and three, where the correct sectors
both have size two. The dimension estimator is then handed a component of size
three whose restricted algebra has rank five, giving $\sqrt5=2.236$, and since
two does not divide three it returns three. It could not have returned the right
answer, because the input was already wrong. The failure is therefore upstream of
the dimension estimator, in the sector split. In every failure we have examined,
here and in the two exceptions noted in Section~\ref{sec:blockid}, the sector
split is what breaks first, which locates the $\varepsilon\approx0.3$ ceiling in
a specific step rather than in the method as a whole. We have examined three such
failures and do not claim this is universal.

The methods lesson is that regular representation validation is structurally
incapable of distinguishing multiplicity from dimension, and that spectral
structure estimators must be tested off it.

\section{Load bearing assumptions}\label{sec:assumptions}
Three assumptions carry the method and are larger than footnotes.

Genericity of the generators. The identification of $\mathcal{A}'$ with the group
algebra image requires the observed generators to generate $\End_G(V)$. Two
generic draws did so in every family tested, on the $S_3$ and $S_4$ regular
representations and on families A and B, consistent with the genericity of pairs
generating a semisimple algebra. We have not characterized when two generators
suffice in general, and structured physical generators, which are sparse and
local, need not be generic. Failure is detectable, since the recovered commutant
dimension comes out too large, but it is not yet handled.

Structural hints. The commutant dimension and the number of isotypic sectors are
supplied to the recovery rather than read from the data. Both appear as
truncation ranks: the commutant is taken as the nullspace of the commutator map
of a given dimension, and the centre as a nullspace of a given dimension within
it. Both are in principle readable from the singular value spectra of those same
maps, and in the noiseless case the separation is exact. Under noise the separation narrows steadily, and by $\varepsilon\approx0.3$ the
gap across the true boundary is no larger than the spacing between adjacent
noise lifted singular values. Reading the ranks rather than supplying
them would therefore impose a noise ceiling of its own, below the
$\varepsilon\approx0.3$ at which the rest of the pipeline still operates, so the
hints are not an incidental convenience. This is the sense, and the only sense,
in which symmetry information enters the method: the group, its representation
and its embedding are withheld and never reconstructed, but the two integers that
fix the shape of the decomposition are given. Characterizing the reading rule and
its own ceiling is left to future work.

Noise model. All robustness numbers use independent perturbations of the
generators from a Gaussian unitary type ensemble. Real operator estimation noise
is correlated, through shared measurement channels and systematic errors, and the
ceiling near $\varepsilon\approx0.3$ may be optimistic off this model. No claim is
made for correlated noise.

\section{Conditions, boundaries and open items}
The construction is synthetic by design; the known symmetry answer keys are what
make the verdicts checkable, and choosing and validating a physical target is a
separate and subsequent question. The demonstrated object is a gated estimator
with low dimensional trained parameters. A full parametric matrix model, in which
the matrices $M_i$ are themselves learned and the symmetry structure emerges in
the fit, is not demonstrated here; it is the immediate next experiment, and until
it runs the claims should be read at the level of the estimator and not the full
model. The singularity at the truth result requires the truth to be exactly
symmetric. The threshold of the dimension estimator is empirical, but its
influence is smaller than that word suggests: across the one hundred and forty
seven isotypic components measured over all families and noise levels, the
restricted algebra singular value count lands exactly on $d^2$ for a divisor $d$
of the component size, so the estimator is not in practice resolving a borderline
case, and the count is in any event insensitive to the threshold fraction over
the range $0.20$ to $0.30$. The noise ceiling near $\varepsilon\approx0.3$ is set
by a specific step: in every failure we have examined, three in total, the centre
based isotypic sector split returns wrong component sizes and the dimension
estimator is handed input on which no correct answer exists, so work aimed at
extending the ceiling should begin there. The $S_4$ noise range is shown to
$\varepsilon=0.05$ and the large $n$ ceiling is not established. The mutual
coincidence of two multiplets of the same dimension remains the hard case,
requiring window or character information. The families off the regular
representation are built from $S_3$ irreps. The objective is piecewise
differentiable only, given the gate. Finally, two distinct forms of machine
dependence affect reported numbers, as recorded in Section~\ref{sec:results}: the
gauge dependent baseline quantities vary across eigensolver builds, which is
Proposition~\ref{prop:sym} appearing in the build system, and seed counts taken
at a classification boundary move by a small number of seeds across builds. All
gauge invariant results reproduce.

\section{Summary}
On synthetic families with genuine nonabelian forced multiplets, under
independent operator estimation noise, with the structure recovered inside the
loop by linear algebra, we obtain the following. Block identification is robust an
order of magnitude past the point where energy clustering dies. The forced versus
accidental gate is structural, since blocks in different isotypic sectors never
merge however close their energies come. Parameter fitting reaches machine
precision in both the gauge ambiguity regime and the genuine singularity regime,
from every initial condition tried, and in the latter the gated gradient does not
merely stay bounded at the truth but vanishes linearly in the distance from it,
so that the point from which the ungated objective is repelled is an attracting
fixed point of the gated one. Observable bias sits at the noise floor against
oracle assisted naive baselines, by a margin of seven standard errors. Beyond the
claimed ceiling the gate degrades gracefully rather than catastrophically: the
first misclassifications appear at a noise level half again above the ceiling,
and where they occur they cost roughly what abandoning the gate entirely would
have cost. Generality is demonstrated off the regular representation, by tests
that the regular representation is structurally unable to run. Each conceptual
risk met during construction was resolved by removing something unnecessary,
namely the homomorphism fit, the spectrum clustering, and two defective
estimators, rather than by adding machinery. Between this proof of concept and a
physical application stand the full parametric matrix model experiment and the
choice of target system.

\section*{Code availability}
The pipeline, the three diagnostic scripts, and the frozen output of the
designated reference run are archived at
\url{https://doi.org/10.5281/zenodo.21808776}, and developed at
\url{https://github.com/rrumabo/commutant_gate}. Every number reported here is
produced by one of those scripts; the mapping from claim to script and output is
given in the repository README.

\end{document}